\documentclass[a4paper,12pt]{amsproc}

\usepackage{amssymb}
\usepackage{graphicx}
\usepackage{float}
\usepackage{amsmath}
\usepackage{array}
\usepackage{multirow}
\usepackage{mathrsfs}
\usepackage{textcomp}
\usepackage[bottom]{footmisc}
\usepackage{xcolor}
\usepackage{cite}

\newcolumntype{M}[1]{>{\raggedright}m{#1}}

\DeclareMathAlphabet{\mathpzc}{OT1}{pzc}{m}{it}

\newtheorem{theorem}{Theorem}[section]
\newtheorem{lemma}[theorem]{Lemma}

\newtheorem{corollary}[theorem]{Corollary}

\theoremstyle{definition}
\newtheorem{definition}[theorem]{Definition}
\newtheorem{example}[theorem]{Example}
\theoremstyle{remark}
\newtheorem{remark}[theorem]{Remark}

\newtheorem{notation}[theorem]{Notation}

\numberwithin{equation}{section}

\allowdisplaybreaks[4]

\begin{document}

\null\vskip -1cm

\title{Bekka's $(c)$-regularity condition and families of line singularities with constant L\^e numbers}

\author{Christophe \textsc{Eyral} and \"Oznur \textsc{Turhan}}

\address{C. Eyral, Institute of Mathematics, Polish Academy of Sciences, ul. \'Sniadeckich~8, 00-656 Warsaw, Poland}  
\email{cheyral@impan.pl} 
\address{\"O. Turhan, Departement of Mathematics, Galatasaray University, Ortaköy 34357, Istanbul, Turkey and Institute of Mathematics, Polish Academy of Sciences, ul. \'Sniadeckich~8, 00-656 Warsaw, Poland}   
\email{oturhan@impan.pl}

\thanks{}



\begin{abstract}
We show that the natural stratifications arising from certain deformation families of line singularities with constant L\^e numbers satisfy Bekka's $(c)$-regularity condition. As a corollary, we obtain that these families are topologically equisingular. Similar results for families of isolated singularities were established by Abderrahmane.
\end{abstract}

\maketitle

\markboth{Christophe Eyral and \"Oznur Turhan}{Bekka's $(c)$-regularity condition and families of line singularities with constant L\^e numbers}  

\section{Introduction}

Let $f(t,\mathbf{z})$ be an analytic function defined in a neighbourhood of the origin in $\mathbb{C} \times \mathbb{C}^n$, and for each parameter $t$, let $V(f_t)$ denote the germ at $\mathbf{0}$ of the hypersurface in $\mathbb{C}^n$ defined by $f_t(\mathbf{z}) := f(t,\mathbf{z})$. It is well known that for $n \neq 3$, if $\{f_t\}$ is a $\mu$-constant family of isolated singularities, then the corresponding family of hypersurfaces $\{V(f_t)\}$ is topologically equisingular (i.e., the local, ambient topological type of $V(f_t)$ at $\mathbf{0}$ is independent of $t$ for all small $t$), as established in the celebrated L\^e–Ramanujam theorem \cite{LR}. Note that a  related result addressing topological \emph{right-equivalence} of functions, and therefore which goes beyond L\^e–Ramanujam's theorem, was subsequently proved by Timourian in \cite{T}. In both cases, \cite{LR} and \cite{T}, the exclusion of the case $n=3$ is a consequence of the use of the $h$-cobordism theorem.

In \cite{M7,M1,M2,M}, Massey raised the question of whether the L\^e–Ramanujam theorem could be generalized to non-isolated singularities, provided that we replace the Milnor number---which is no longer relevant in this case---by the L\^e numbers. See also Vannier \cite{V1,V2}. In \cite{B}, Fern\'andez de Bobadilla answered this question positively if $n\geq 5$ and each hypersurface $V(f_t)$ has a $1$-dimensional singular locus.
On the other hand, in \cite{Bo1}, he gave a counterexample for the case of a hypersurface with a critical set of dimension $3$. It remains unknown whether the $\mu$-constancy in the case $n=3$ for isolated singularities and the constancy of L\^e numbers in the case $n\leq 4$ for hypersurfaces with $1$-dimensional singular sets are sufficient conditions for topological equisingularity. Hereafter, we shall refer to this question as the \emph{L\^e–Ramanujam problem}, whether for hypersurfaces with isolated singularities or for those whose singular locus is one-dimensional.

For isolated singularities, there are partial results due to Oka \cite{O,O1} and Abderrahmane \cite{A}. They concern Newton non-degenerate singularities, which form an important class of singularities introduced by Kouchnirenko in \cite{K}. More precisely, in \cite{O,O1}, Oka showed that if $f$ defines a $\mu$-constant family $\{f_t\}$ of Newton non-degenerate isolated hypersurface singularities of the form $f_t(\mathbf{z})=f_0(\mathbf{z})+tz_1^{\alpha_1}\cdots z_n^{\alpha_n}$, then the family $\{f_t\}$ is topologically equisingular. Notably, this result holds without requiring that the Newton diagram of $f_t$ is independent of $t$ or that $n\neq 3$. 

In \cite{A}, Abderrahmane generalized Oka's result to all $\mu$-constant families of Newton non-degenerate functions $f_t$ with an isolated singularity at $\mathbf{0}$. In fact, Abderrahmane proved a stronger result. He showed that if $\{f_t\}$ is a $\mu$-constant family of \emph{convenient}, Newton non-degenerate, isolated singularities, then $V(f)\setminus \Sigma f$ is \emph{Bekka $(c)$-regular} over $\Sigma f=\mathbb{C}\times \{\mathbf{0}\}$, where $V(f)$ stands for the hypersurface of $\mathbb{C}\times\mathbb{C}^n$ defined by $f$, and $\Sigma f$ denotes its singular locus. In other words, the natural partition $(V(f)\setminus\Sigma f) \sqcup \Sigma f$ is a Bekka $(c)$-regular stratification of $V(f)$. The $(c)$-regularity condition, introduced by Bekka in \cite {Bekka}, is a slightly less rigid condition than the classical Whitney's $(b)$-regularity condition, in the sense that any $(b)$-regular stratified set is also $(c)$-regular. However, just as Whitney's $(b)$-regular stratifications, Bekka's $(c)$-regular stratifications are also locally topologically trivial along the strata. This latter result, which extends the classical Thom--Mather theorem, was established by Bekka in \cite{Bekka}.
So, in particular, in the above setting, if $V(f)\setminus \Sigma f$ is Bekka $(c)$-regular over $\Sigma f$, then the family of hypersurfaces $\{V(f_t)\}$ is topologically equisingular. 

The assumption that the functions $f_t$ are \emph{convenient} means that the Newton diagram of each $f_t$ intersects all coordinate axes. While this condition plays a critical role in establishing Bekka's $(c)$-regularity---serving to define a suitable  \emph{control function}---it is not required, as rightly noted in \cite{A}, to deduce that the family $\{V(f_t)\}$ is topologically equisingular. Indeed, it is well known that when $f_t$ has an isolated singularity, adding terms of the form $z_i^N$ for sufficiently large $N$ does not alter the embedded topological type of $V(f_t)$.

In the present paper, we extend Abderrahmane's result to certain families of \emph{line singularities}. This new framework requires substituting not only the constancy of the Milnor number with that of the L\^e numbers, but also replacing convenient functions with \emph{quasi-convenient} ones (see Definition \ref{def-adm}). Indeed, by Kouchnirenko's theorem \cite{K}, under the Newton non-degeneracy condition, convenient function-germs always have an isolated singularity. Moreover, in contrast with isolated singularities---which possess neighbourhoods consisting entirely of smooth points, except the singularity itself---a line singularity contains infinitely many singular points that accumulate near it, thereby giving rise to a markedly more intricate topological structure. This added complexity poses substantial challenges for analysis. To manage these difficulties, we narrow our focus to specific families of line singularities, which we refer to henceforth as \emph{admissible families} (see Definition \ref{def-adm}). Our main theorem then asserts that for any admissible family $\{f_t\}$, if the L\^e numbers of the $f_t$'s are independent of $t$ for all sufficiently small $t$, then the natural stratification $(V(f)\setminus\Sigma f) \sqcup \Sigma f$ of $V(f)$ is Bekka $(c)$-regular (see Theorem \ref{mt}).

As a corollary of this and of Bekka's theorem, we conclude that any admissible family with constant L\^e numbers is topologically equisingular (see Corollary \ref{cor}). Note that we neither require the Newton diagram to be independent of $t$ nor assume $n \geq 5$, thereby partially answering the L\^e--Ramanujam problem for our class of hypersurfaces with $1$-dimensional singular sets.
Note that this corollary also extends, in a certain direction, a theorem of Damon \cite{D}, which concerns families of Newton non-degenerate singularities with a \emph{fixed} Newton boundary---on the other hand, Damon's result does not require quasi-convenience and it includes hypersurfaces having singular sets of higher dimensions.

\begin{notation}
Throughout, we use the following notation.
For a function 
\begin{equation*}
f\colon \mathbb{C}\times \mathbb{C}^n \to \mathbb{C}
\end{equation*}
of the variables $(t,\mathbf{z})=(t,z_1,\ldots, z_n)$, we denote by $\partial f$ the gradient of $f$, that is, 
\begin{equation*}
\partial f:=(\overline{\partial f/\partial t},\overline{\partial f/\partial z_1},\ldots,\overline{\partial f/\partial z_n}). 
\end{equation*}
Here, the bar denotes the complex conjugacy. The gradient of $f$ with respect to the variables $\mathbf{z}$ (respectively, the variables $t,z_1$) will be denoted by $\partial_{\mathbf{z}} f$ (respectively, by $\partial_{t,z_1} f$). 

For two functions $f,g\colon \mathbb{C}^n\to\mathbb{C}$ of the variables $\mathbf{z}=(z_1,\ldots,z_n)$, the expression 
\begin{equation*}
\vert f(\mathbf{z})\vert \ll \vert g(\mathbf{z})\vert
\quad\mbox{as}\quad \mathbf{z}\to\mathbf{z}_0
\end{equation*}
means $\lim_{\mathbf{z}\to\mathbf{z}_0} f(\mathbf{z}) / g(\mathbf{z})=0$, while the relation 
\begin{equation*}
\vert f(\mathbf{z})\vert \lesssim \vert g(\mathbf{z})\vert
\end{equation*}
 means that there exists a  constant $K>0$ such that $\vert f(\mathbf{z})\vert \leq K\cdot \vert g(\mathbf{z})\vert$.

For any $\mathbf{z}=(z_1,\ldots,z_n)\in\mathbb{C}^n$ and any $\alpha=(a_1,\ldots,a_n)\in\mathbb{N}^n$, we write 
\begin{equation*}
\mathbf{z}^{\alpha}:=z_1^{a_1}\cdots z_n^{a_n}
\quad\mbox{and}\quad
\bar{\mathbf{z}}^{\alpha}:=\bar z_1^{a_1}\cdots \bar z_n^{a_n}, 
\end{equation*}
where, again, the bar stands for the complex conjugacy. It will also be useful, for a given $\mathbf{z}=(z_1,z_2,\ldots,z_n)$, to use the notation $\tilde{\mathbf{z}}:=(z_2,\ldots,z_n)$. Finally, in order to highlight the relevant variables, we shall occasionally use the notation $\mathbb{C}_{t}\times\mathbb{C}_{\mathbf{z}}^n$ or $\mathbb{C}_{z_1}\times \mathbb{C}_{\tilde{\mathbf{z}}}^{n-1}$, etc.
\end{notation}

\section{Bekka's $(c)$-regularity condition}

\begin{definition}[\mbox{see \cite[\S 2]{Bekka}}]
We say that a stratification\footnote{Here, by a \emph{stratification}, we mean a locally finite partition of $A$ into submanifolds that satisfies the \emph{frontier condition}, that is, if $X,Y\in\mathcal{A}$ are such that $Y\cap\bar X\not=\emptyset$, then $Y\subseteq \bar X$. Here, by $\bar X$, we mean the closure of $X$ in the ambient manifold $M$.} $\mathcal{A}$ of a closed subset $A$ in a manifold $M$ is \emph{Bekka $(c)$-regular} if for any stratum $Y\in \mathcal{A}$, there exists an open neighbourhood $U_Y$ of $Y$ in $M$ together with a $C^1$-function $\rho_Y\colon U_Y\to [0,\infty)$ (called a \emph{control function}) such that $\rho_Y^{-1}(0)=Y$ and the restricted map
\begin{equation*}
\rho_Y\colon U_Y\cap\mbox{star}(Y)\to [0,\infty)
\end{equation*}
is a \emph{Thom map}, that is, for any point $y\in Y$ and any stratum $X\in\mbox{star}(Y)$, the following inclusion holds:
\begin{equation}\label{rcr}
\lim_{X\ni x\to y} T_x((\rho_Y\vert_X)^{-1}(\rho_Y(x))) \supseteq T_yY.
\end{equation}
\end{definition}

Here, by $\mbox{star}(Y)$, we mean the set of all strata $X$ having $Y$ in their closures. As usual, $T_yY$ denotes the tangent space of $Y$ at $y$. Similarly for $T_x((\rho_Y\vert_X)^{-1}(\rho_Y(x)))$. 

\begin{example}[\mbox{see \cite{Trotman}}]
Whitney $(b)$-regular stratifications are Bekka $(c)$-regular, while the converse is false.
\end{example}

\begin{example}[\mbox{see \cite{Bekka}}]
Consider the Brian\c con--Speder family \cite{BS}, which is given by
\begin{equation*}
f(t,z_1,z_2,z_3):=z_3^5+tz_2^6z_3+z_2^7z_1+z_1^{15},
\end{equation*}
 and look at the corresponding stratification 
\begin{equation*}
(V(f)\setminus (\mathbb{C}\times \{\mathbf{0}\})) \sqcup (\mathbb{C}\times \{\mathbf{0}\})
\end{equation*}
 of $V(f)$. It is well known that this stratification does not satisfy the Whitney $(b)$-regularity condition (see \cite{BS}). However, it is Bekka $(c)$-regular (see \cite{Bekka}). More generally, in \cite{Bekka}, Bekka observes that the canonical stratification obtained from a family of weighted homogeneous polynomials with isolated singularities and constant Milnor number is always Bekka $(c)$-regular (see also Abderrahmane \cite[Remark 3.4]{A}).
\end{example}

\section{Admissible families of line singularities}

In this section, we set up the geometric framework in which we will operate and recall some well-known basic facts that may prove useful later on.

\subsection{Families of line singularities}
Let $(t,\mathbf{z})=(t,z_1,\ldots,z_n)$ be linear coordinates for $\mathbb{C}\times \mathbb{C}^n$, and let 
\begin{equation*}\label{lfi}
f\colon \mathbb{C} \times \mathbb{C}^n \rightarrow \mathbb{C},\
(t,\mathbf{z})\mapsto f(t,\mathbf{z}), 
\end{equation*}
be an analytic function defined in a neighbourhood of the origin $(0,\mathbf{0})\in\mathbb{C} \times \mathbb{C}^n$.
As usual, we write $f_t(\mathbf{z}):= f(t,\mathbf{z})$, and we denote by $V(f_t)$ the hypersurface in $\mathbb{C}^n$ defined by the equation $f_t(\mathbf{z})=0$. We assume that $f_t(\mathbf{0})=f(t,\mathbf{0})=0$, so that the origin $\mathbf{0}\in \mathbb{C}^n$ belongs to $V(f_t)$ for all small $t$. 

\begin{definition}
 As in \cite[\S 4]{M7}, we say that $f$ defines a family of hypersurfaces $V(f_t)$ with \emph{line singularities} if for each $t$ small enough, the following two conditions hold:
\begin{enumerate}
\item
the singular locus $\Sigma f_t$ of $V(f_t)$ near $\mathbf{0}$ is given by the $z_1$-axis of $\mathbb{C}^n$;
\item
the restriction of $f_t$ to the hyperplane $V(z_1)$ defined by $z_1=0$ has an isolated critical point at $\mathbf{0}$.
\end{enumerate}
\end{definition}

If $f=\{f_t\}$ is a family of line singularities, then, by \cite[Remark 1.29]{M}, the partition 
\begin{equation*}
(V(f_t)\setminus\Sigma f_t) \sqcup (\Sigma f_t\setminus\{\mathbf{0}\}) \sqcup \{\mathbf{0}\}
\end{equation*}
 of $V(f_t)$ defines a \emph{good stratification} and the hyperplane $V(z_1)$ is a \emph{prepolar slice} of $f_t$ at~$\mathbf{0}$ with respect to this stratification.
In particular, combined with \cite[Proposition~1.23]{M}, this implies that the \emph{L\^e numbers} 
\begin{equation*}
\lambda^0_{f_t,\mathbf{z}} 
\quad\mbox{and}\quad\lambda^1_{f_t,\mathbf{z}} 
\end{equation*}
 of $f_t$ at~$\mathbf{0}$ with respect to the coordinates $\mathbf{z}$ do exist.\footnote{For the definitions and basic properties of good stratifications, prepolarity and L\^e numbers, we refer the reader to \cite{M}. Note that for line singularities, the only possible non-zero L\^e numbers are precisely $\lambda^0_{f_t,\mathbf{z}} $ and $\lambda^1_{f_t,\mathbf{z}} $; all the other L\^e numbers $\lambda^k_{f_t,\mathbf{z}} $ for $2\leq k\leq n-1$ exist and are equal to zero (see \cite{M}).}
Note that the L\^e number $\lambda^1_{f_t,\mathbf{z}}$ is equal to the generic Milnor number $\mathring{\mu}(f_t)$, that is, the Milnor number $\mu(f_t\vert_{V(z_1-a_1)})$ of $f_t\vert_{V(z_1-a_1)}$ at $\mathbf{0}\in\mathbb{C}^{n-1}$ for a sufficiently small $a_1\not=0$.\footnote{The number $\mu(f_t\vert_{V(z_1-a_1)})$ is independent of $a_1$ if $a_1$ is sufficiently small and non-zero.}

\subsection{Admissible families}\label{subsect-AF}

Let $f=\{f_t\}$ be a family of line singularities. Denote by $\Gamma_+(f_t)$ the Newton polyhedron of $f_t$ with respect to the coordinates $\mathbf{z}$, and write $\Gamma(f_t)$ for its Newton boundary, that is, the union of all compact faces of $\Gamma_+(f_t)$ (see \cite{K}). Then, let
\begin{equation*}
\alpha_1,\ldots,\alpha_m\in\mathbb{N}^{n}
\end{equation*}
denote those vertices of $\Gamma(f_{t})$, for $t\not=0$, which correspond to monomials of the form 
\begin{equation*}
z_2^{a_2}\cdots z_n^{a_n}, 
\end{equation*}
that is, each $\alpha_i$ is written as $\alpha_i=(0,a_{i2},\ldots,a_{in})$ for $1\leq i\leq m$. Finally, for each $2\leq j\leq n$, define 
\begin{equation*}
a_j:=\sup_{1\leq i\leq m} a_{ij}.
\end{equation*}

\begin{definition}\label{def-adm}
We say that a family $f=\{f_t\}$ of line singularities is \emph{admissible} if for all small $t$ the following three conditions are satisfied:
\begin{enumerate}
\item[(i)]
$f_t$ is \emph{Newton non-degenerate} (in the sense of Kouchnirenko \cite{K});
\item[(ii)]
if $t\not=0$, then $f_t$ is \emph{quasi-convenient}, that is, it contains a term of the form $c_j(t) z_j^{b_j}$ for each $j\not=1$, where $c_j(t)$ is a non-zero coefficient;
\item[(iii)]
for any monomial of $f_t$ of the form $c(t)z_1^{b_1}z_2^{b_2}\cdots z_n^{b_n}$ with $b_1\not=0$, we have $b_j\geq a_j$ for all $2\leq j\leq n$, while $b_1$ can take any positive value.
\end{enumerate}
\end{definition}

Note that if $f_t$ is Newton non-degenerate and quasi-convenient, then it  cannot contain a term of the form $c_1(t)z_1^{a_1}$, as this would imply that it is convenient, and therefore, with an isolated singularity at $\mathbf{0}$. Also, observe that if $f_0$ is quasi-convenient, so is $f_t$ for all small $t$. 
Finally, note that for each $j\not=1$, the smallest integer $b_j$ that appears in (ii) determines a  vertex $(0,\ldots,0,b_j,0,\ldots,0)$ of the Newton boundary $\Gamma(f_t)$. 

\begin{example}
The family $f=\{f_t\}$ given by 
\begin{equation*}
f_t(z_1,z_2,z_3):=z_2^2+z_3^2+tz_1^2z_2^2z_3^2
\end{equation*}
satisfies conditions (ii) and (iii). We easily check that $f_t$ is Newton non-degenerate and that the singular locus $\Sigma f_t$ of the corresponding hypersurface $V(f_t)$ is given by the $z_1$-axis of $\mathbb{C}^3$. Moreover, since  the origin is an isolated critical point of the restricted maps $f_t\vert_{V(z_1)}$, the $f_t$'s form an \emph{admissible family} of line singularities.
\end{example}

\section{Statements of the results}

Again, let $f=\{f_t\}$ be a family of line singularities.
To state our main result, we first observe that if the L\^e numbers $\lambda^0_{f_t,\mathbf{z}} $ and $\lambda^1_{f_t,\mathbf{z}} $ are independent of $t$, then, by \cite[Theorem 6.5]{M}, the $t$-axis (namely, $\mathbb{C}\times\{\mathbf{0}\}$) satisfies \emph{Thom's $(a_f)$ condition} at the origin with respect to the ambient stratum. This means that if $\{(t_k,\mathbf{z}_k)\}$ is a sequence of points in $(\mathbb{C} \times \mathbb{C}^n)\setminus \Sigma f$ such that $(t_k,\mathbf{z}_k)\to (0,\mathbf{0})$ and $T_{(t_k,\mathbf{z}_k)}V(f-f(t_k,\mathbf{z}_k))\to T$, then $\mathbb{C}\times\{\mathbf{0}\}=T_{(0,\mathbf{0})}(\mathbb{C}\times\{\mathbf{0}\})\subseteq T$. In particular, this implies that the singular locus $\Sigma f$ of the hypersurface $V(f)\subseteq \mathbb{C}\times\mathbb{C}^n$ is smooth and given by 
\begin{equation*}
\Sigma f=\mathbb{C}_t\times\mathbb{C}_{z_1}\times\{\mathbf{0}\}=\mathbb{C}_t\times(\mbox{$z_1$-axis of $\mathbb{C}^n$})
\end{equation*}
 (see \cite[Lemma 44]{B}). 

Our main result is stated as follows.

\begin{theorem}\label{mt}
Assume that $f=\{f_t\}$ is an admissible family of line singularities. If the L\^e numbers $\lambda^0_{f_t,\mathbf{z}} $ and $\lambda^1_{f_t,\mathbf{z}} $ are independent of $t$ for all sufficiently small $t$, then the natural stratification $(V(f)\setminus \Sigma f) \sqcup\Sigma f$ of the closed subset $V(f)\subseteq\mathbb{C}\times\mathbb{C}^n$ is Bekka $(c)$-regular. 
\end{theorem}

\begin{remark}\label{strat-extension}
If the stratification $\{V(f)\setminus \Sigma f,\Sigma f\}$ of the hypersurface $V(f)$ is Bekka $(c)$-regular, then it naturally extends to a Bekka $(c)$-regular stratification of the whole space $\mathbb{C}\times \mathbb{C}^{n}$ by adding the open stratum $(\mathbb{C}\times\mathbb{C}^{n})\setminus V(f)$.
\end{remark}

Theorem \ref{mt} extends to \emph{admissible families of line singularities} a similar result by Abderrahmane  which concerns families of \emph{convenient Newton non-degenerate isolated singularities} (see Theorem~1.1, together with its proof, and the end of the introduction in \cite{A}). Let us emphasize that quasi-convenience---just as convenience in the case of isolated singularities---serves solely to guarantee the existence of a suitable control function $\rho$, specifically ensuring that $\rho^{-1}(0) = \Sigma f$.

Combined with Bekka's generalization of the Thom--Mather theorem (see \cite[\S 3, Th\'eo\-r\`eme 1]{Bekka}) and a result of King \cite[Theorem~1]{King}, Theorem \ref{mt} implies the following corollary, which provides a version of Abderrahmane's result \cite[Theorem 1.1]{A} for admissible families of line singularities.

\begin{corollary}\label{cor}
Again, assume that $f=\{f_t\}$ is an admissible family of line singularities such that the L\^e numbers $\lambda^0_{f_t,\mathbf{z}} $ and $\lambda^1_{f_t,\mathbf{z}} $ are independent of $t$ for all small $t$. Then the family of hypersurfaces $\{V(f_t)\}$ is topologically equisingular.
\end{corollary}

\begin{proof}
By Theorem \ref{mt} and Remark \ref{strat-extension}, the stratification 
\begin{equation*}
((\mathbb{C}\times \mathbb{C}^n)\setminus V(f)) \sqcup (V(f)\setminus \Sigma f) \sqcup\Sigma f
\end{equation*} 
 of $\mathbb{C}\times \mathbb{C}^n$ is Bekka $(c)$-regular. Therefore, by Bekka's generalization of Thom--Mather's theorem (see \cite[\S 3, Th\'eor\`eme 1]{Bekka}), $V(f)$ is locally topologically trivial along $\Sigma f$. In other words, the $2$-parameter family $\{V(f_{t,z_1})\}_{t,z_1}$ of hypersurface-germs  $V(f_{t,z_1})\subseteq\mathbb{C}^{n-1}$ is topologically equisingular. In particular, for any small enough fixed $t_0$, the $1$-parameter family $\{V(f_{t_0,z_1})\}_{z_1}$ is topologically equisingular too. 
Thus, by a result of King \cite[Theorem 1]{King}\footnote{For $n\not=4$, we may alternatively invoke Timourian's result \cite[Theorem p.~437]{T}.}, the family $\{f_{t_0,z_1}\}_{z_1}$ of functions is topologically trivial with respect to the \emph{right-equivalence}. That is, for each fixed $z_1^0$, there exist a neighbourhood $U\subseteq \mathbb{C}_{z_1}\times\mathbb{C}^{n-1}_{\tilde{\mathbf{z}}}$ of $(z_1^0,\mathbf{0})$, and neighbourhoods $D\subseteq \mathbb{C}_{z_1}$ and $W\subseteq \mathbb{C}^{n-1}_{\tilde{\mathbf{z}}}$ of $z_1^0$ and $\mathbf{0}$, respectively, together with a homeomorphism $h\colon U\to D\times W$ such that
\begin{equation*}
h(V(f_{t_0})\cap U)=D\times (V(f_{t_0,z_1^0})\cap W).
\end{equation*} 
In particular, this holds for $z_1^0=0$. Similarly, for another small $t'_0$, there also exists such a local, ambient homeomorphism sending the hypersurface-germ $V(f_{t'_0})$ onto the germ at $(0,\mathbf{0})$ of $\mathbb{C}\times V(f_{t'_0,0})$. Since the germs $(V(f_{t_0,0}),\mathbf{0})$ and $(V(f_{t'_0,0}),\mathbf{0})$ are homeomorphic via an ambient homeomorphism of $\mathbb{C}^{n-1}$ for sufficiently small values of $t_0$ and $t'_0$, it follows that
$(V(f_{t_0}),\mathbf{0})$ and $(V(f_{t'_0}),\mathbf{0})$
 are likewise homeomorphic via an ambient homeomorphism of $\mathbb{C}^n$.
\end{proof}

Note that Corollary \ref{cor} applies to any admissible families of line singularities with constant L\^e numbers in $\mathbb{C}^n$ ($n\geq 2$), without any other restriction on $n$, and hence partially answering the L\^e--Ramanujam problem for our class of hypersurfaces with $1$-dimensional singular sets.
Also, observe that Corollary \ref{cor} includes families with possibly \emph{non-constant} Newton boundaries, thereby extending, in a specific direction, the scope of a result by Damon \cite[p.~273]{D}, which addresses families of Newton non-degenerate singularities—whether or not they are admissible or even $1$-dimensional---with a \emph{fixed} Newton boundary.
Here, note that if $\Gamma(f_t)$ is independent of $t$ for all small $t$, including $t=0$, and if $f_t$ is Newton non-degenerate for all such $t$'s, then, by \cite[Corollary 5.2]{EOR}, the L\^e numbers $\lambda^0_{f_t,\mathbf{z}}$ and $\lambda^1_{f_t,\mathbf{z}}$ are independent of $t$. A related result addressing \emph{Whitney equisingularity} is presented in \cite[Theorem 3.8]{EO}, applicable to families of Newton non-degenerate functions that are ``uniformly locally tame'' (see \cite[Definition 3.7]{EO}) and maintain a constant ``non-compact Newton boundary'' (in particular, a constant Newton boundary).

\begin{example}
We easily check that the family $f=\{f_t\}$ given by 
\begin{equation*}
f(t,z_1,z_2,z_3) = z_2^4+z_3^3+t\, z_1^2 z_2^4 z_3^3
\end{equation*}
forms an admissible family of line singularities such that $\lambda^0_{f_t,\mathbf{z}}=0$ and $\lambda^1_{f_t,\mathbf{z}}=6$ for all small $t$. Then, by Theorem \ref{mt} and Remark \ref{strat-extension}, the stratification 
\begin{equation*}
\{(\mathbb{C}\times\mathbb{C}^n)\setminus V(f),V(f)\setminus \Sigma f,\Sigma f\}
\end{equation*}
 is Bekka $(c)$-regular, and by Corollary \ref{cor}, the family $\{V(f_t)\}$ is topologically equisingular.
\end{example}

\begin{example}
Now, look at the family $f=\{f_t\}$ given by
\begin{equation*}\label{ex2}
f(t,z_1,z_2,z_3) = z_2^2+z_3^2+z_1z_2^2z_3^2+t\, z_1 z_2^2. 
\end{equation*}
Again, we readily verify that this defines an admissible family of line singularities. Regarding the Lê numbers, a direct computation yields $\lambda^0_{f_t,\mathbf{z}}=0$ and $\lambda^1_{f_t,\mathbf{z}}=1$ for all small $t$.
Consequently, Theorem~\ref{mt} and Corollary~\ref{cor} are once again applicable.
\end{example}

\section{Proof of Theorem \ref{mt}}

The proof presented below is inspired by that of \cite[Theorem~1.1]{A}.

Consider the control function $\rho\colon \mathbb{C}\times\mathbb{C}^n\to\mathbb{C}$ defined by
\begin{equation*}
(t,\mathbf{z})\mapsto
\rho(t,\mathbf{z}):=\sum_{i=1}^m\mathbf{z}^{\alpha_i}\bar{\mathbf{z}}^{\alpha_i},
\end{equation*}
where, let us recall, $\alpha_1,\ldots,\alpha_m\in\mathbb{N}^{n}$ are the vertices of $\Gamma(f_{t})$, $t\not=0$, corresponding to monomials of the form $z_2^{a_2}\cdots z_n^{a_n}$ (see \S \ref{subsect-AF}).
As the $f_t$'s, $t\not=0$, are quasi-convenient and do not contain any term of the form $c_1(t)z_1^{a_1}$ (see ibid.), we~have
\begin{equation*}
\rho^{-1}(0)=\mathbb{C}_t\times\mathbb{C}_{z_1}\times\{\mathbf{0}\}=\Sigma f. 
\end{equation*}

To prove the theorem, it is enough to show that the pair of strata $(V(f)\setminus\Sigma f,\Sigma f)$ is $(c)$-regular at $(0,\mathbf{0})$ with respect to $\rho$, which means that the relation \eqref{rcr} holds true for $Y=\Sigma f$, $X=V(f)\setminus\Sigma f$ and $y=(0,\mathbf{0})$. Indeed, by \cite[Theorem 2.4]{BK}, $(X,Y)$ is Bekka $(c)$-regular at $y$ if and only if it is Whitney $(a)$-regular at $y$ and satisfies the condition $(m)$ of \cite[Definition 2.3]{BK}. As the latter condition is globally defined in a neighbourhood of $Y$, and since the Whitney $(a)$-regularity holds at each point outside a strict closed analytic subset of $Y$ (see \cite[Lemma~19.3]{Whitney}), it follows that if $(X,Y)$ is Bekka $(c)$-regular at $(0,\mathbf{0})$, then it is also Bekka $(c)$-regular at each point in a neighbourhood of $(0,\mathbf{0})$ in $Y$.

\subsection*{Step 0 -- Strategy of the proof} 

Consider the multi-parameter deformation $g(\mathbf{s},t,\mathbf{z})$ of $f_0$ defined by
\begin{equation*}
(\mathbf{s},t,\mathbf{z})\in \mathbb{C}^m\times\mathbb{C}\times\mathbb{C}^n
\mapsto g(\mathbf{s},t,\mathbf{z}):=f(t,\mathbf{z})+\sum_{i=1}^m s_i\, \mathbf{z}^{\alpha_i},
\end{equation*} 
and put $g_{\mathbf{s},t}(\mathbf{z}):=g(\mathbf{s},t,\mathbf{z})$. (Of course, $\mathbf{s}=(s_1,\ldots,s_m)$.) As usual,  denote by $V(g_{\mathbf{s},t})$ the corresponding  hypersurface of $\mathbb{C}^n$, and  write $\Sigma g_{\mathbf{s},t}$ for its singular locus. 
Then the proof decomposes into three steps.
In Step 1, we show that $\dim_{\mathbf{0}}\Sigma g_{\mathbf{s},t}\leq 1$, and that the L\^e numbers $\lambda^0_{g_{\mathbf{s},t},\mathbf{z}}$ and $\lambda^1_{g_{\mathbf{s},t},\mathbf{z}}$ 
of $g_{\mathbf{s},t}$ at $\mathbf{0}$ with respect to the coordinates $\mathbf{z}$ exist and remain constant for all $(\mathbf{s},t)$ sufficiently close to $(\mathbf{0},0)$. By a general
version of the L\^e--Saito theorem for multi-parameter families of non-isolated  singularities due to Massey \cite{M}, this implies that $\mathbb{C}_{\mathbf{s}}^m \times \mathbb{C}_t \times \{\mathbf{0}\}$ satisfies Thom's $(a_g)$ condition. In step~2, we observe that condition (iii) of Definition \ref{def-adm} implies
\begin{equation*}
\vert \partial_{t,z_1} f(t,\mathbf{z})\vert \lesssim
\sum_{i=1}^m \vert \mathbf{z}^{\alpha_i}\vert
\end{equation*}
near the origin---a relation which, combined with the Thom $(a_g)$ condition, implies 
\begin{equation*}
\vert \partial_{t,z_1} f(t,\mathbf{z})\vert \ll
\frac{\vert\partial f(t,\mathbf{z})\wedge \partial_{t,\mathbf{z}}\rho(t,\mathbf{z})\vert}{\vert\partial_{t,\mathbf{z}}\rho(t,\mathbf{z})\vert} 
\end{equation*}
as $V(f)\setminus \Sigma f \ni (t,\mathbf{z})\to (0,\mathbf{0})$. Finally, in Step 3, we conclude that this relation, in turn, implies
\begin{align*}
\frac{\vert\langle\partial_{t,\mathbf{z}}\rho(t,\mathbf{z}), \partial f(t,\mathbf{z})\rangle\vert} {\vert\partial f(t,\mathbf{z})\vert^2}
\, \vert\partial_{t,z_1} f(t,\mathbf{z})\vert & \ll 
\vert \partial_{t,\mathbf{z}}\rho_{\vert V(f)\setminus \Sigma f}(t,\mathbf{z})\vert
\quad\mbox{and}\\
\vert \partial_{t,z_1} f(t,\mathbf{z})\vert & \ll 
\vert \partial_{\tilde{\mathbf{z}}} f(t,\mathbf{z})\vert
\end{align*}
as $V(f)\setminus \Sigma f \ni (t,\mathbf{z})\to (0,\mathbf{0})$---two conditions which, by another theorem of Abderrahmane (see \cite[Theorem 1]{A2}), are equivalent to saying that $V(f)\setminus \Sigma f$ is Bekka $(c)$-regular over $\Sigma f$ at the point $(0,\mathbf{0})$ with respect to the control function $\rho$.

\subsection*{Step 1 -- Constancy of the L\^e numbers $\lambda^0_{g_{\mathbf{s},t},\mathbf{z}}$ and $\lambda^1_{g_{\mathbf{s},t},\mathbf{z}}$} 

First of all, we show that $\lambda^0_{g_{\mathbf{s},t},\mathbf{z}}$ and $\lambda^1_{g_{\mathbf{s},t},\mathbf{z}}$ are properly defined.

\begin{lemma}\label{NNDd1}
There exists a neighbourhood $W$ of $(\mathbf{0},0)$ in $\mathbb{C}_\mathbf{s}^m\times\mathbb{C}_t$ such that for any $(\mathbf{s},t)\in W$, the following two properties hold true:
\begin{enumerate}
\item
$\dim_{\mathbf{0}}\Sigma g_{\mathbf{s},t}\leq 1$, where $\dim_{\mathbf{0}}\Sigma g_{\mathbf{s},t}$ is the dimension  of $\Sigma g_{\mathbf{s},t}$ at $\mathbf{0}$;
\item
the L\^e numbers $\lambda^0_{g_{\mathbf{s},t},\mathbf{z}} $ and $\lambda^1_{g_{\mathbf{s},t},\mathbf{z}} $ exist.
\end{enumerate}
\end{lemma}

\begin{proof}
Essentially, this follows from the Iomdine--L\^e--Massey formula \cite[Theorem~4.5]{M} and \cite[Proposition 3.1]{MS}.
Indeed, since $g_{\mathbf{0},0}\equiv f_0$ has a line singularity at~$\mathbf{0}\in\mathbb{C}^n$, the restriction map $g_{\mathbf{0},0}\vert_{V(z_1)}\equiv f_0\vert_{V(z_1)}$ has an isolated singularity at $\mathbf{0}\in\mathbb{C}^{n-1}$ (by definition). Besides, by the Iomdine--L\^e--Massey formula, for any $a\geq 2+\lambda^0_{f_0,\mathbf{z}}$ the function 
\begin{equation*} 
g_{\mathbf{0},0}+z_1^{a}\equiv f_0+z_1^{a}
\end{equation*} 
 also has an isolated singularity at~$\mathbf{0}\in\mathbb{C}^n$. Then, since the families
\begin{equation*}
\{g_{\mathbf{s},t}\vert_{V(z_1)}\}_{\mathbf{s},t}
\quad\mbox{and}\quad
\{g_{\mathbf{s},t}+z_1^{a}\}_{\mathbf{s},t}
\end{equation*} 
are deformations of $g_{\mathbf{0},0}\vert_{V(z_1)}$ and $g_{\mathbf{0},0}+z_1^{a}$, respectively, there exists a neighbourhood $W$ of $(\mathbf{0},0)$ in $\mathbb{C}_\mathbf{s}^m\times\mathbb{C}_t$ such that for any $(\mathbf{s},t)\in W$, the functions
\begin{equation*}
g_{\mathbf{s},t}\vert_{V(z_1)}
\quad\mbox{and}\quad
g_{\mathbf{s},t}+z_1^{a}
\end{equation*} 
both have an isolated singularity at the origin too (see, e.g., \cite[Chap.~I, Theorem~2.6]{GLS}).
In particular, the condition $\dim_{\mathbf{0}}g_{\mathbf{s},t}\vert_{V(z_1)}=0$ implies that $\dim_{\mathbf{0}}\Gamma^1_{g_{\mathbf{s},t},\mathbf{z}}\leq 1$, and hence, by \cite[Proposition 1.23]{M}, the L\^e number $\lambda^0_{g_{\mathbf{s},t},\mathbf{z}}$ is defined for any $(\mathbf{s},t)\in W$.
Here, $\Gamma^1_{g_{\mathbf{s},t},\mathbf{z}}$ denotes the $1$st (relative) polar variety of $g_{\mathbf{s},t}$ with respect to $\mathbf{z}$ (see \cite[Definition 1.3]{M}).

Now, we show the assertion on the dimension. We argue by contradiction. Suppose that $W$ contains a point $(\mathbf{s}_0,t_0)$ such that $\dim_{\mathbf{0}}\Sigma g_{\mathbf{s}_0,t_0}> 1$. Then, by the (proof of the) Iomdine--L\^e--Massey formula, for any $a\geq 2+\lambda^0_{g_{\mathbf{s}_0,t_0},\mathbf{z}}$ we have
\begin{equation*}
\dim_{\mathbf{0}} \Sigma (g_{\mathbf{s}_0,t_0}+z_1^{a})=
\dim_{\mathbf{0}} \Sigma g_{\mathbf{s}_0,t_0}-1>0.
\end{equation*} 
In particular, combined with the previous observations, this implies that for any $a\geq \sup\{2+\lambda^0_{f_0,\mathbf{z}},2+\lambda^0_{g_{\mathbf{s}_0,t_0},\mathbf{z}}\}$, the function $g_{\mathbf{s}_0,t_0}+z_1^{a}$ has an isolated singularity at $\mathbf{0}$ on one hand, and $\dim_{\mathbf{0}} \Sigma (g_{\mathbf{s}_0,t_0}+z_1^{a})>0$ on the other hand---a contradiction. 

Finally, we observe that for all $(\mathbf{s},t)\in W$, the L\^e number $\lambda^1_{g_{\mathbf{s},t},\mathbf{z}} $ exists. This follows immediately from \cite[Remark 1.29]{M}, knowing that
$\dim_{\mathbf{0}}\Sigma g_{\mathbf{s},t}\leq 1$ and $g_{\mathbf{s},t}\vert_{V(z_1)}$ has an isolated critical point at the origin.
\end{proof}

\begin{remark}
In fact, we can easily check that for any $(\mathbf{s},t)\in W$, the $z_1$-axis of $\mathbb{C}_{\mathbf{z}}^n$ is contained in $\Sigma g_{\mathbf{s},t}$, and hence, by Lemma \ref{NNDd1}, we have $\dim_{\mathbf{0}}\Sigma g_{\mathbf{s},t} = 1$.
\end{remark}

The proof of Lemma \ref{NNDd1} shows that there is a neighbourhood $W$ of $(\mathbf{0},0)$ in $\mathbb{C}_\mathbf{s}^m\times\mathbb{C}_t$ such that for any $(\mathbf{s},t)\in W$, the function $g_{\mathbf{s},t} + z_1^a$ has an isolated singularity at~$\mathbf{0}$ if $a\geq 2+\lambda^0_{f_0,\mathbf{z}}$. The next lemma establishes the independence of the Milnor number of $g_{\mathbf{s},t} + z_1^a$ for large $a$'s. The proof relies on the \emph{Newton non-degeneracy} of the $f_t$'s. This lemma will then be used to prove the constancy of the L\^e numbers $\lambda^0_{g_{\mathbf{s},t},\mathbf{z}}$ and $\lambda^1_{g_{\mathbf{s},t},\mathbf{z}}$.

\begin{lemma}\label{inter}
For any $a\geq 2+\lambda^0_{f_0,\mathbf{z}}$ sufficiently large, the Milnor number $\mu(g_{\mathbf{s},t}+z_1^a)$ of $g_{\mathbf{s},t}+z_1^a$ at $\mathbf{0}$ is independent of $(\mathbf{s},t)$ for all $(\mathbf{s},t)\in W$, where $W$ is as in Lemma \ref{NNDd1}.
\end{lemma}

\begin{proof}
Clearly, for all $(\mathbf{s},t)\in W$ such that $t\neq 0$, the Newton boundaries $\Gamma(g_{\mathbf{s},t}+z_1^a)$ and $\Gamma(f_t+z_1^a)$ of $g_{\mathbf{s},t}+z_1^a$ and $f_t+z_1^a$ coincide, and thus so do their \emph{Newton numbers}\footnote{For the definition, see \cite{K}.}:
\begin{equation}\label{ref3}
\nu(g_{\mathbf{s},t}+z_1^a)=\nu(f_t+z_1^a).
\end{equation} 
By the upper-semicontinuity of the Milnor number $\mu(g_{\mathbf{s},t}+z_1^a)$ with respect to the variables $(\mathbf{s},t)$, we can assume, by shrinking $W$ if needed, that for all $(\mathbf{s},t)\in W$, we have
\begin{equation}\label{ref1}
\mu(f_0+z_1^a)\equiv \mu(g_{\mathbf{0},0}+z_1^a)\geq \mu(g_{\mathbf{s},t}+z_1^a).
\end{equation} 
Moreover, by a theorem of Kouchnirenko \cite[\S 1.10, Th\'eor\`eme I]{K}, we also have 
\begin{equation}\label{ref2}
\mu(g_{\mathbf{s},t}+z_1^a)\geq\nu(g_{\mathbf{s},t}+z_1^a).
\end{equation} 

Now, given that $a\geq 2+\lambda^0_{f_0,\mathbf{z}}=2+\lambda^0_{f_t,\mathbf{z}}$, and
possibly after further reducing $W$, it follows from the uniform Iomdine--L\^e--Massey formula \cite[Theorem 4.15]{M} that
\begin{equation*}
\mu(f_t+z_1^a)=\lambda^0_{f_t,\mathbf{z}} +(a-1)\lambda^1_{f_t,\mathbf{z}}, 
\end{equation*} 
and therefore,
\begin{equation}\label{refa1}
\mu(f_t+z_1^a)=\mu(f_0+z_1^a)
\end{equation} 
for all $t$ such that $(\mathbf{s},t)\in W$.
In other words, the family of isolated singularities $\{f_t+z_1^a\}$ is $\mu$-constant.
Moreover, since $f_t$ is \emph{Newton non-degenerate}, we could have chosen the integer $a$ sufficiently large from the beginning, so that, in addition, the function $f_t+z_1^a$ is Newton non-degenerate too for all $t$ with $(\mathbf{s},t)\in W$ (see \cite[Lemma~3.7]{BO})\footnote{Note that in \cite[Lemma~3.7]{BO}, the lower bound for $a$ depends only on the Newton boundary $\Gamma(f_t)$. In our case, we have two such bounds: one for $\Gamma(f_0)$ and one for $\Gamma(f_t)$, $t\not=0$. We take $a$ greater than the largest one.}. Thus, by the second part of Kouchnirenko's theorem \cite[\S 1.10, Th\'eor\`eme~I]{K}, for all such $t$ we have
\begin{equation}\label{ref4}
\mu(f_t+z_1^a)=\nu(f_t+z_1^a).
\end{equation} 

All of this shows that for any $(\mathbf{s},t)\in W$:
\begin{equation*}
\mu(f_0+z_1^a)\stackrel{\eqref{ref1}}{\geq} \mu(g_{\mathbf{s},t}+z_1^a)\stackrel{\eqref{ref2}}{\geq} \nu(g_{\mathbf{s},t}+z_1^a)\stackrel{\eqref{ref3}}{=}\nu(f_t+z_1^a)\stackrel{\eqref{ref4}}{=}\mu(f_t+z_1^a)\stackrel{\eqref{refa1}}{=}\mu(f_0+z_1^a).
\end{equation*} 
This implies that the family $\{g_{\mathbf{s},t}+z_1^a\}_{(\mathbf{s},t)\in W}$ is $\mu$-constant.
\end{proof}

Now, we are able to prove the constancy of the L\^e numbers $\lambda^0_{g_{\mathbf{s},t},\mathbf{z}}$~and~$\lambda^1_{g_{\mathbf{s},t},\mathbf{z}}$.

\begin{lemma}\label{lemmaclng}
There exists a neighbourhood $U\subseteq W$ of $(\mathbf{0},0)$ in $\mathbb{C}_{\mathbf{s}}^m\times\mathbb{C}_t$  such that for any $(\mathbf{s},t)\in U$, the L\^e numbers $\lambda^0_{g_{\mathbf{s},t},\mathbf{z}} $ and $\lambda^1_{g_{\mathbf{s},t},\mathbf{z}}$ are independent of $(\mathbf{s},t)$. (Here again, $W$ is as in Lemma \ref{NNDd1}.)
\end{lemma}

\begin{proof}
By \cite[Corollary 4.16]{M}, the pair of L\^e numbers 
$(\lambda^1_{g_{\mathbf{s},t},\mathbf{z}} ,\lambda^0_{g_{\mathbf{s},t},\mathbf{z}} )$ is lexicographically upper-semicontinuous in the variables $(\mathbf{s},t)$, that is, there exists a neighbourhood $U\subseteq W$ of $(\mathbf{0},0)$ in $\mathbb{C}_{\mathbf{s}}^m\times\mathbb{C}_t$ such that:
\begin{enumerate}
\item[$\bullet$] 
either $\lambda^1_{g_{\mathbf{0},0},\mathbf{z}} >\lambda^1_{g_{\mathbf{s},t},\mathbf{z}} $
for all $(\mathbf{s},t)\in U$;
\item[$\bullet$] 
or $\lambda^1_{g_{\mathbf{0},0},\mathbf{z}} =\lambda^1_{g_{\mathbf{s},t},\mathbf{z}} $ and $\lambda^0_{g_{\mathbf{0},0},\mathbf{z}} \geq \lambda^0_{g_{\mathbf{s},t},\mathbf{z}} $ for all $(\mathbf{s},t)\in U$.
\end{enumerate}
If $\lambda^1_{g_{\mathbf{0},0},\mathbf{z}} >\lambda^1_{g_{\mathbf{s},t},\mathbf{z}} $
for all $(\mathbf{s},t)\in U$, then, in particular, $\lambda^1_{f_0,\mathbf{z}} \equiv \lambda^1_{g_{\mathbf{0},0},\mathbf{z}} >\lambda^1_{g_{\mathbf{0},t},\mathbf{z}} \equiv\lambda^1_{f_t,\mathbf{z}} $, which is a contradiction. Thus, necessarily, 
\begin{equation}\label{ref5}
\lambda^1_{g_{\mathbf{0},0},\mathbf{z}} =\lambda^1_{g_{\mathbf{s},t},\mathbf{z}} 
\quad\mbox{and}\quad
\lambda^0_{g_{\mathbf{0},0},\mathbf{z}} \geq \lambda^0_{g_{\mathbf{s},t},\mathbf{z}} 
\end{equation} 
 for all $(\mathbf{s},t)\in U$. 
We claim that $\lambda^0_{g_{\mathbf{0},0},\mathbf{z}} =\lambda^0_{g_{\mathbf{s},t},\mathbf{z}} $ for all $(\mathbf{s},t)\in U$. Indeed, pick any  $(\mathbf{s},t)\in U$. For $a>0$ sufficiently large (so that, in particular, we have $a\geq 2+\lambda^0_{f_0,\mathbf{z}} \equiv 2+\lambda^0_{g_{\mathbf{0},0},\mathbf{z}} \geq 2+\lambda^0_{g_{\mathbf{s},t},\mathbf{z}}$),  the Iomdine--L\^e--Massey formula \cite[Theorem 4.5]{M} shows that
\begin{align*}
& \mu(g_{\mathbf{0},0}+z_1^a)=\lambda^0_{g_{\mathbf{0},0},\mathbf{z}} +(a-1)\lambda^1_{g_{\mathbf{0},0},\mathbf{z}} ,\\
& \mu(g_{\mathbf{s},t}+z_1^a)=\lambda^0_{g_{\mathbf{s},t},\mathbf{z}} +(a-1)\lambda^1_{g_{\mathbf{s},t},\mathbf{z}} .
\end{align*} 
Since the family $\{g_{\mathbf{s},t}+z_1^a\}_{(\mathbf{s},t)\in U}$ is $\mu$-constant---after possibly increasing $a$, as ensured by Lemma \ref{inter}---and given that $\lambda^1_{g_{\mathbf{0},0},\mathbf{z}}=\lambda^1_{g_{\mathbf{s},t},\mathbf{z}}$ (by the first part of \eqref{ref5}), it follows that $\lambda^0_{g_{\mathbf{0},0},\mathbf{z}}=\lambda^0_{g_{\mathbf{s},t},\mathbf{z}}$.
\end{proof}

\subsection*{Step 2 -- Controlling $\vert \partial_{t,z_1} f(t,\mathbf{z})\vert$}

Combined with the general version of the L\^e--Saito theorem for multi-parameter families of non-isolated singularities due to Massey (see \cite[Theorem~6.9]{M}),  Lemma \ref{lemmaclng} above implies that $\mathbb{C}_{\mathbf{s}}^m\times \mathbb{C}_t\times\{\mathbf{0}\}$ satisfies \emph{Thom's $(a_g)$ condition} at the origin with respect to the ambient stratum. More precisely, this means that if $\mathbf{x}_k:=(\mathbf{s}_k,t_k,\mathbf{z}_k)$, $k\in\mathbb{N}$, is a sequence of points in $(\mathbb{C}_{\mathbf{s}}^m\times \mathbb{C}_t\times\mathbb{C}_{\mathbf{z}}^{n})\setminus \Sigma g$ such that 
\begin{equation*}
\mathbf{x}_k\to \mathbf{0}
\quad\mbox{and}\quad
T_{\mathbf{x}_k}V(g-g(\mathbf{x}_k))\to T,
\end{equation*}
then $\mathbb{C}_{\mathbf{s}}^m\times \mathbb{C}_t\times\{\mathbf{0}\}\subseteq T$.
In other words,
\begin{equation}\label{formula1}
\vert \partial_{t} f(t,\mathbf{z})\vert + \sum_{i=1}^m \vert \mathbf{z}^{\alpha_i}\vert  
\ll \bigg\vert\partial_{\mathbf{z}} f(t,\mathbf{z}) + \sum_{i=1}^m s_i\, \partial_{\mathbf{z}}(\mathbf{z}^{\alpha_i})\bigg\vert 
\leq \bigg\vert\partial f(t,\mathbf{z}) + \sum_{i=1}^m s_i\, \partial_{t,\mathbf{z}}(\mathbf{z}^{\alpha_i})\bigg\vert
\end{equation}
as $(\mathbf{s},t,\mathbf{z})\in (\mathbb{C}_{\mathbf{s}}^m\times \mathbb{C}_t\times\mathbb{C}_{\mathbf{z}}^{n})\setminus \Sigma g$ approaches the origin. 
Then, by the same argument to that given in the proof of \cite[Lemma 3.3]{A} (see also \cite{Parusinski}), we deduce from this relation that
\begin{equation}\label{r1}
\sum_{i=1}^m \vert \mathbf{z}^{\alpha_i}\vert \ll 
\inf_{\eta\in\mathbb{C}}
\bigg\vert\partial f(t,\mathbf{z}) + \sum_{i=1}^m\eta\,\bar{\mathbf{z}}^{\alpha_i} \partial_{t,\mathbf{z}}(\mathbf{z}^{\alpha_i})\bigg\vert = \inf_{\eta\in\mathbb{C}} \bigg\vert \partial f(t,\mathbf{z}) + \eta\, \partial_{t,\mathbf{z}}\rho(t,\mathbf{z})\bigg\vert
\end{equation}
as $V(f)\setminus \Sigma f\ni (t,\mathbf{z})\to (0,\mathbf{0})$. 

Now, by condition (iii) of Definition \ref{def-adm}, we have
\begin{equation}\label{r2}
\vert \partial_{t,z_1} f(t,\mathbf{z})\vert \lesssim \sum_{i=1}^m \vert \mathbf{z}^{\alpha_i}\vert
\end{equation}
near the origin. Indeed, let us first show that $\vert \partial_{z_1} f(t,\mathbf{z})\vert \lesssim \sum_{i=1}^m \vert \mathbf{z}^{\alpha_i}\vert$.
We argue by contradiction. Expand $\partial_{z_1} f(t,\mathbf{z})$ as 
\begin{equation*}
\partial_{z_1} f(t,\mathbf{z})=\sum_{\beta=(b_1,\ldots,b_n)} c_\beta(t)z_1^{b_1-1}z_2^{b_2}\cdots z_n^{b_n},
\end{equation*}
 and by the curve selection lemma, suppose that there exists an analytic curve 
\begin{equation*}
s\in [0,\varepsilon)\mapsto (t(s),\mathbf{z}(s))=(t(s),z_1(s),\ldots,z_n(s))
\end{equation*}
such that $(t(s),\mathbf{z}(s))\to (0,\mathbf{0})$ and
\begin{equation*}
\bigg\vert\sum_{\beta=(b_1,\ldots,b_n)}  \frac{c_\beta(t(s)) z_1(s)^{b_1-1}z_2(s)^{b_2}\cdots z_n(s)^{b_n}}{\sum_{i=1}^m \vert z_2(s)^{a_{i2}}\cdots z_n(s)^{a_{in}}\vert}\bigg\vert\to\infty
\end{equation*}
as $s\to 0$. Then, necessarily, there exists $a_{ij}$ ($j\geq 2$) such that $a_{ij}>b_j$, which contradicts the condition (iii) of Definition \ref{def-adm}.
To establish that a similar inequality holds for $\vert \partial_{t} f(t,\mathbf{z})\vert$, we proceed analogously. As before, the terms in $\vert \partial_{t} f(t,\mathbf{z})\vert$ involving the variable $z_1$ are governed by condition (iii) of Definition \ref{def-adm}. The remaining terms, which do not contain $z_1$, may lie either above or on the Newton boundary $\Gamma(f_t)$. Those not situated on $\Gamma(f_t)$ are asymptotically negligible in comparison to the dominant monomials $\mathbf{z}^{\alpha_i}$ located on $\Gamma(f_t)$. The terms positioned on $\Gamma(f_t)$ are asymptotically equivalent to one of the monomials $\mathbf{z}^{\alpha_i}$.

Combined with \eqref{r1}, the inequality \eqref{r2} implies
\begin{equation}\label{mr3}
\begin{aligned}
\vert \partial_{t,z_1} f(t,\mathbf{z})\vert 
& \stackrel{\text{\eqref{r2}}}{\lesssim} \sum_{i=1}^m \vert \mathbf{z}^{\alpha_i}\vert \\
& \stackrel{\text{\eqref{r1}}}{\ll}
\inf_{\eta\in\mathbb{C}} \bigg\vert \partial f(t,\mathbf{z}) + \eta\, \partial_{t,\mathbf{z}}\rho(t,\mathbf{z})\bigg\vert = \frac{\vert\partial f(t,\mathbf{z})\wedge \partial_{t,\mathbf{z}}\rho(t,\mathbf{z})\vert}{\vert\partial_{t,\mathbf{z}}\rho(t,\mathbf{z})\vert} 
\end{aligned}
\end{equation}
as $V(f)\setminus \Sigma f\ni (t,\mathbf{z})\to (0,\mathbf{0})$.

\subsection*{Step 3 -- Conclusion}

The above relation implies that the pair of strata 
\begin{equation*}
(V(f)\setminus \Sigma f,\Sigma f)
\end{equation*}
 is Bekka $(c)$-regular at $(0,\mathbf{0})$ with respect to the control function $\rho$. The argument is completely similar to that used by Abderrahmane in \cite[Theorem~2.1]{A}. 
Indeed, put $X:=V(f)\setminus \Sigma f$, and as in \cite{A}, decompose $\partial_{t,\mathbf{z}}\rho(t,\mathbf{z})$ as 
\begin{equation*}
\partial_{t,\mathbf{z}}\rho(t,\mathbf{z})=\partial_{t,\mathbf{z}}\rho_{\vert X}(t,\mathbf{z})+\partial_{t,\mathbf{z}}\rho_{\vert N}(t,\mathbf{z}),
\end{equation*}
where $N$ denotes the normal space to $X$. The discussion at the bottom of page 169 in \cite{A} shows that
\begin{equation}\label{Ad}
\partial_{t,\mathbf{z}}\rho(t,\mathbf{z}) = \partial_{t,\mathbf{z}}\rho_{\vert X}(t,\mathbf{z}) + A(t,\mathbf{z}) \partial f(t,\mathbf{z}),
\end{equation}
so that
\begin{equation}\label{Ad2}
\partial_{t,\mathbf{z}}\rho_{\vert X}(t,\mathbf{z})=-A(t,\mathbf{z})
\bigg(\partial_{t} f(t,\mathbf{z}), \partial_{z_1} f(t,\mathbf{z}), -\frac{\partial_{\tilde{\mathbf{z}}}\rho(t,\mathbf{z})}{A(t,\mathbf{z})}+\partial_{\tilde{\mathbf{z}}}f(t,\mathbf{z})\bigg),
\end{equation}
where $A(t,\mathbf{z}):=\langle\partial_{t,\mathbf{z}}\rho(t,\mathbf{z}), \partial f(t,\mathbf{z})\rangle / \vert\partial f(t,\mathbf{z})\vert^2$. From \eqref{Ad}, we deduce that
\begin{equation}\label{Ad3}
\begin{aligned}
\vert\partial_{t,\mathbf{z}}\rho_{\vert X}(t,\mathbf{z})\vert^2 
& = \frac{\vert \partial_{t,\mathbf{z}}\rho(t,\mathbf{z})\vert^2\, \vert \partial f(t,\mathbf{z})\vert^2 - \vert \langle\partial_{t,\mathbf{z}}\rho(t,\mathbf{z}), \partial f(t,\mathbf{z})\rangle \vert^2}{\vert\partial f(t,\mathbf{z})\vert^2} \\
& = \frac{\vert\partial f(t,\mathbf{z})\wedge \partial_{t,\mathbf{z}}\rho(t,\mathbf{z})\vert^2}{\vert\partial f(t,\mathbf{z})\vert^2}.
\end{aligned}
\end{equation}
It is then straightforward to check that, when combined with \eqref{mr3}, the relations \eqref{Ad2} and \eqref{Ad3} yield to the new relations
\begin{equation*}
\vert A(t,\mathbf{z})\vert \cdot \vert\partial_{t,z_1} f(t,\mathbf{z})\vert \ll \vert \partial_{t,\mathbf{z}}\rho_{\vert X}(t,\mathbf{z})\vert
\quad\mbox{and}\quad
\vert \partial_{t,z_1} f(t,\mathbf{z})\vert\ll \vert \partial_{\tilde{\mathbf{z}}} f(t,\mathbf{z})\vert
\end{equation*}
as $X\ni (t,\mathbf{z})\to (0,\mathbf{0})$. Now, by \cite[Theorem 1]{A2}, these two relations are equivalent to saying that $X:=V(f)\setminus \Sigma f$ is Bekka $(c)$-regular over $\Sigma f$ at the point $(0,\mathbf{0})$ with respect to the control function $\rho$.

\bibliographystyle{amsplain}

\end{document}